\documentclass{amsart}
\usepackage{amsfonts}

\newtheorem{theorem}{Theorem}
\theoremstyle{plain}

\newtheorem{lemma}{Lemma}

\newtheorem{remark}{Remark}

\numberwithin{equation}{section}
\input{tcilatex}

\begin{document}
\title[H\"{o}lder integral inequality]{New Refinements for integral and sum
forms of H\"{o}lder inequality}
\author{\.{I}mdat \.{I}\c{s}can}
\address{Department of Mathematics, Faculty of Arts and Sciences,\\
Giresun University, 28200, Giresun, Turkey.}
\email{imdati@yahoo.com, imdat.iscan@giresun.edu.tr}

\begin{abstract}
In this paper, new refinements for integral and sum forms of H\"{o}lder
inequality are established. We note that many existing inequalities related
to the H\"{o}lder inequality can be improved via obtained new inequalities
in here, we show this in an application
\end{abstract}

\maketitle

\section{Introduction}

\bigskip The famous Young inequality for two scalars is the $t$-weighted
arithmetic-geometric mean inequality. This inequality says that if $x,y>0$
and $t\in \lbrack 0,1]$, then%
\begin{equation}
x^{t}y^{1-t}\leq tx+(1-t)y  \label{0-1}
\end{equation}%
with equality if and only if $a=b.$ Let $p,q>1$ such that $1/p+1/q=1$. The
inequality (\ref{0-1}) can be written as%
\begin{equation}
xy\leq \frac{x^{p}}{p}+\frac{y^{q}}{q}  \label{0-2}
\end{equation}%
for any $x,y\geq 0$. In this form, the inequality (\ref{0-2}) was used to
prove the celebrated H\"{o}lder inequality. One of the most important
inequalities of analysis is H\"{o}lder's inequality. It contributes wide
area of pure and applied mathematics and plays a key role in resolving many
problems in social science and cultural science as well as in natural
science.

\begin{theorem}[H\"{o}lder Inequality for Integrals \protect\cite{MPF13}]
Let $p>1$ and $1/p+1/q=1$. If $f\ $and $g$ are real functions defined on $%
\left[ a,b\right] $ and if $\left\vert f\right\vert ^{p},\left\vert
g\right\vert ^{q}$ are integrable functions on $\left[ a,b\right] $ then%
\begin{equation}
\int_{a}^{b}\left\vert f(x)g(x)\right\vert dx\leq \left(
\int_{a}^{b}\left\vert f(x)\right\vert ^{p}dx\right) ^{1/p}\left(
\int_{a}^{b}\left\vert g(x)\right\vert ^{q}dx\right) ^{1/q},  \label{1-1}
\end{equation}%
with equality holding if and only if $A\left\vert f(x)\right\vert
^{p}=B\left\vert g(x)\right\vert ^{q}$ almost everywhere, where $A$ and $B$
are constants.
\end{theorem}

\begin{theorem}[H\"{o}lder Inequality for Sums \protect\cite{MPF13}]
Let $a=\left( a_{1},...,a_{n}\right) $ and $b=\left( b_{1},...,b_{n}\right) $
be two positive n-tuples and $p,q>0$ such that $1/p+1/q=1.$ Then we have%
\begin{equation}
\sum_{k=1}^{n}a_{k}b_{k}\leq \left( \sum_{k=1}^{n}a_{k}^{p}\right)
^{1/p}\left( \sum_{k=1}^{n}b_{k}^{q}\right) ^{1/q}.  \label{1-2}
\end{equation}%
Equality hold in (\ref{1-2}) if and only if $a^{p}$ and $b^{q}$ are
proportional.
\end{theorem}

Of course the H\"{o}lder's inequality has been extensively explored and
tested to a new situation by a number of scientists. Many generalizations
and refinements for H\"{o}lder's inequality have been obtained so far. See,
for example, \cite{BB61,DE68,HLP52,KB16,KY12,MP90,MPF13,QH11,W07,Y03} and
the references therein. In this paper, by using a simple proof method some
new refinements for integral and sum forms of H\"{o}lder's inequality are
obtained.

\section{Main Results}

\begin{theorem}
\label{2.2} Let $p>1$ and $1/p+1/q=1$. If $f\ $and $g$ are real functions
defined on $\left[ a,b\right] $ and if $\left\vert f\right\vert
^{p},\left\vert g\right\vert ^{q}$ are integrable functions on $\left[ a,b%
\right] $, then

\begin{itemize}
\item[i.)] 
\begin{eqnarray}
&&\int_{a}^{b}\left\vert f(x)g(x)\right\vert dx  \label{2-1} \\
&\leq &\frac{1}{b-a}\left\{ \left( \int_{a}^{b}(b-x)\left\vert
f(x)\right\vert ^{p}dx\right) ^{1/p}\left( \int_{a}^{b}(b-x)\left\vert
g(x)\right\vert ^{q}dx\right) ^{1/q}\right.  \notag \\
&&\left. +\left( \int_{a}^{b}(x-a)\left\vert f(x)\right\vert ^{p}dx\right)
^{1/p}\left( \int_{a}^{b}(x-a)\left\vert g(x)\right\vert ^{q}dx\right)
^{1/q}\right\}  \notag
\end{eqnarray}

\item[ii.)] 
\begin{eqnarray}
&&\frac{1}{b-a}\left\{ \left( \int_{a}^{b}(b-x)\left\vert f(x)\right\vert
^{p}dx\right) ^{1/p}\left( \int_{a}^{b}(b-x)\left\vert g(x)\right\vert
^{q}dx\right) ^{1/q}\right.  \notag \\
&&\left. +\left( \int_{a}^{b}(x-a)\left\vert f(x)\right\vert ^{p}dx\right)
^{1/p}\left( \int_{a}^{b}(x-a)\left\vert g(x)\right\vert ^{q}dx\right)
^{1/q}\right\}  \notag \\
&\leq &\left( \int_{a}^{b}\left\vert f(x)\right\vert ^{p}dx\right)
^{1/p}\left( \int_{a}^{b}\left\vert g(x)\right\vert ^{q}dx\right) ^{1/q}.
\label{2-2}
\end{eqnarray}
\end{itemize}
\end{theorem}

\begin{proof}
i.)\emph{First method for Proof (Short method):} By using of H\"{o}lder
inequality in (\ref{1-1}), it is easily seen that 
\begin{eqnarray*}
&&\int_{a}^{b}\left\vert f(x)g(x)\right\vert dx \\
&=&\frac{1}{b-a}\left\{ \int_{a}^{b}\left\vert
(b-x)^{1/p}f(x)(b-x)^{1/q}g(x)\right\vert dx\right. \\
&&\left. +\int_{a}^{b}\left\vert (x-a)^{1/p}f(x)(x-a)^{1/q}g(x)\right\vert
dx\right\} \\
&\leq &\frac{1}{b-a}\left\{ \left( \int_{a}^{b}(b-x)\left\vert
f(x)\right\vert ^{p}dx\right) ^{1/p}\left( \int_{a}^{b}(b-x)\left\vert
g(x)\right\vert ^{q}dx\right) ^{1/q}\right. \\
&&\left. +\left( \int_{a}^{b}(x-a)\left\vert f(x)\right\vert ^{p}dx\right)
^{1/p}\left( \int_{a}^{b}(x-a)\left\vert g(x)\right\vert ^{q}dx\right)
^{1/q}\right\} .
\end{eqnarray*}
\emph{Second method for Proof (Long method):} Applying (\ref{1-1}) on the
subinterval $[a,\lambda b+(1-\lambda )a]$ and on the subinterval $[\lambda
b+(1-\lambda )a,b]$, respectively, we get%
\begin{equation*}
\int_{a}^{\lambda b+(1-\lambda )a}\left\vert f(x)g(x)\right\vert dx\leq
\left( \int_{a}^{\lambda b+(1-\lambda )a}\left\vert f(x)\right\vert
^{p}dx\right) ^{1/p}\left( \int_{a}^{\lambda b+(1-\lambda )a}\left\vert
g(x)\right\vert ^{q}dx\right) ^{1/q}
\end{equation*}%
and%
\begin{equation*}
\int_{\lambda b+(1-\lambda )a}^{b}\left\vert f(x)g(x)\right\vert dx\leq
\left( \int_{\lambda b+(1-\lambda )a}^{b}\left\vert f(x)\right\vert
^{p}dx\right) ^{1/p}\left( \int_{\lambda b+(1-\lambda )a}^{b}\left\vert
g(x)\right\vert ^{q}dx\right) ^{1/q}.
\end{equation*}%
Adding the resulting inequalities, we get:%
\begin{eqnarray}
\int_{a}^{b}\left\vert f(x)g(x)\right\vert dx &\leq &\left(
\int_{a}^{\lambda b+(1-\lambda )a}\left\vert f(x)\right\vert ^{p}dx\right)
^{1/p}\left( \int_{a}^{\lambda b+(1-\lambda )a}\left\vert g(x)\right\vert
^{q}dx\right) ^{1/q}  \notag \\
&&+\left( \int_{\lambda b+(1-\lambda )a}^{b}\left\vert f(x)\right\vert
^{p}dx\right) ^{1/p}\left( \int_{\lambda b+(1-\lambda )a}^{b}\left\vert
g(x)\right\vert ^{q}dx\right) ^{1/q}.  \label{2-1a}
\end{eqnarray}%
By the change of variable $x=ub+(1-u)a$; on the right hand sides integrals
in (\ref{2-1a}), we have%
\begin{eqnarray*}
&&\int_{a}^{b}\left\vert f(x)g(x)\right\vert dx \\
&\leq &\left( b-a\right) \left\{ \left( \int_{0}^{\lambda }\left\vert
f(ub+(1-u)a)\right\vert ^{p}du\right) ^{1/p}\left( \int_{0}^{\lambda
}\left\vert g(ub+(1-u)a)\right\vert ^{q}du\right) ^{1/q}\right. \\
&&\left. +\left( \int_{\lambda }^{1}\left\vert f(ub+(1-u)a)\right\vert
^{p}du\right) ^{1/p}\left( \int_{\lambda }^{1}\left\vert
g(ub+(1-u)a)\right\vert ^{q}du\right) ^{1/q}\right\} .
\end{eqnarray*}%
Integrating both sides of this inequality over $\left[ 0,1\right] $ with
respect to $\lambda $ we obtain that 
\begin{eqnarray*}
&&\int_{a}^{b}\left\vert f(x)g(x)\right\vert dx \\
&\leq &\left( b-a\right) \left\{ \int_{0}^{1}\left( \int_{0}^{\lambda
}\left\vert f(ub+(1-u)a)\right\vert ^{p}du\right) ^{1/p}\left(
\int_{0}^{\lambda }\left\vert g(ub+(1-u)a)\right\vert ^{q}du\right)
^{1/q}d\lambda \right. \\
&&\left. +\int_{0}^{1}\left( \int_{\lambda }^{1}\left\vert
f(ub+(1-u)a)\right\vert ^{p}du\right) ^{1/p}\left( \int_{\lambda
}^{1}\left\vert g(ub+(1-u)a)\right\vert ^{q}du\right) ^{1/q}d\lambda
\right\} .
\end{eqnarray*}%
Then, By applying of H\"{o}lder inequality for the right hand sides
integrals in the last inequality, we have%
\begin{eqnarray*}
&&\int_{a}^{b}\left\vert f(x)g(x)\right\vert dx \\
&\leq &\left( b-a\right) \left\{ \left( \int_{0}^{1}\int_{0}^{\lambda
}\left\vert f(ub+(1-u)a)\right\vert ^{p}dud\lambda \right) ^{1/p}\left(
\int_{0}^{1}\int_{0}^{\lambda }\left\vert g(ub+(1-u)a)\right\vert
^{q}dud\lambda \right) ^{1/q}\right. \\
&&\left. +\left( \int_{0}^{1}\int_{\lambda }^{1}\left\vert
f(ub+(1-u)a)\right\vert ^{p}dud\lambda \right) ^{1/p}\left(
\int_{0}^{1}\int_{\lambda }^{1}\left\vert g(ub+(1-u)a)\right\vert
^{q}dud\lambda \right) ^{1/q}\right\} .
\end{eqnarray*}%
By Fubini theorem and the change of variable $u=(x-a)/(b-a)$ we get%
\begin{eqnarray*}
&&\int_{a}^{b}\left\vert f(x)g(x)\right\vert dx \\
&\leq &\left( b-a\right) \left\{ \left( \int_{0}^{1}\int_{u}^{1}\left\vert
f(ub+(1-u)a)\right\vert ^{p}d\lambda du\right) ^{1/p}\left(
\int_{0}^{1}\int_{u}^{1}\left\vert g(ub+(1-u)a)\right\vert ^{q}d\lambda
du\right) ^{1/q}\right. \\
&&\left. +\left( \int_{0}^{1}\int_{0}^{u}\left\vert f(ub+(1-u)a)\right\vert
^{p}d\lambda du\right) ^{1/p}\left( \int_{0}^{1}\int_{0}^{u}\left\vert
g(ub+(1-u)a)\right\vert ^{q}d\lambda du\right) ^{1/q}\right\}
\end{eqnarray*}%
\begin{eqnarray*}
&=&\left( b-a\right) \left\{ \left( \int_{0}^{1}\left( 1-u\right) \left\vert
f(ub+(1-u)a)\right\vert ^{p}du\right) ^{1/p}\left( \int_{0}^{1}\left(
1-u\right) \left\vert g(ub+(1-u)a)\right\vert ^{q}du\right) ^{1/q}\right. \\
&&\left. +\left( \int_{0}^{1}u\left\vert f(ub+(1-u)a)\right\vert
^{p}du\right) ^{1/p}\left( \int_{0}^{1}u\left\vert g(ub+(1-u)a)\right\vert
^{q}du\right) ^{1/q}\right\} \\
&=&\frac{1}{b-a}\left\{ \left( \int_{a}^{b}(b-x)\left\vert f(x)\right\vert
^{p}dx\right) ^{1/p}\left( \int_{a}^{b}(b-x)\left\vert g(x)\right\vert
^{q}dx\right) ^{1/q}\right. \\
&&\left. +\left( \int_{a}^{b}(x-a)\left\vert f(x)\right\vert ^{p}dx\right)
^{1/p}\left( \int_{a}^{b}(x-a)\left\vert g(x)\right\vert ^{q}dx\right)
^{1/q}\right\} .
\end{eqnarray*}

b) First let us consider the case%
\begin{equation*}
\left( \int_{a}^{b}\left\vert f(x)\right\vert ^{p}dx\right) ^{1/p}\left(
\int_{a}^{b}\left\vert g(x)\right\vert ^{q}dx\right) ^{1/q}=0.
\end{equation*}

Then, $f(x)=0$ for almost everywhere $x\in \left[ a,b\right] $ or $g(x)=0$
for almost every where $x\in \left[ a,b\right] .$ Thus, we have 
\begin{equation*}
\int_{a}^{b}\left\vert f(x)g(x)\right\vert dx=0.
\end{equation*}%
Therefore the inequality (\ref{2-2}) is trivial in this case.

Finally, we consider the case%
\begin{equation*}
I=\left( \int_{a}^{b}\left\vert f(x)\right\vert ^{p}dx\right) ^{1/p}\left(
\int_{a}^{b}\left\vert g(x)\right\vert ^{q}dx\right) ^{1/q}\neq 0.
\end{equation*}%
Then%
\begin{eqnarray*}
&&\frac{1}{\left( b-a\right) I}\left\{ \left( \int_{a}^{b}(b-x)\left\vert
f(x)\right\vert ^{p}dx\right) ^{1/p}\left( \int_{a}^{b}(b-x)\left\vert
g(x)\right\vert ^{q}dx\right) ^{1/q}\right. \\
&&\left. +\left( \int_{a}^{b}(x-a)\left\vert f(x)\right\vert ^{p}dx\right)
^{1/p}\left( \int_{a}^{b}(x-a)\left\vert g(x)\right\vert ^{q}dx\right)
^{1/q}\right\} \\
&\leq &\frac{1}{b-a}\left\{ \left( \frac{\int_{a}^{b}(b-x)\left\vert
f(x)\right\vert ^{p}dx}{\int_{a}^{b}\left\vert f(x)\right\vert ^{p}dx}%
\right) ^{1/p}\left( \frac{\int_{a}^{b}(b-x)\left\vert g(x)\right\vert ^{q}dx%
}{\int_{a}^{b}\left\vert g(x)\right\vert ^{q}dx}\right) ^{1/q}\right. \\
&&\left. +\left( \frac{\int_{a}^{b}(x-a)\left\vert f(x)\right\vert ^{p}dx}{%
\int_{a}^{b}\left\vert f(x)\right\vert ^{p}dx}\right) ^{1/p}\left( \frac{%
\int_{a}^{b}(x-a)\left\vert g(x)\right\vert ^{q}dx}{\int_{a}^{b}\left\vert
g(x)\right\vert ^{q}dx}\right) ^{1/q}\right\} .
\end{eqnarray*}%
Applying (\ref{0-1}) on the right hand sides integrals of the last inequality%
\begin{eqnarray*}
&&\frac{1}{\left( b-a\right) I}\left\{ \left( \int_{a}^{b}(b-x)\left\vert
f(x)\right\vert ^{p}dx\right) ^{1/p}\left( \int_{a}^{b}(b-x)\left\vert
g(x)\right\vert ^{q}dx\right) ^{1/q}\right. \\
&&\left. +\left( \int_{a}^{b}(x-a)\left\vert f(x)\right\vert ^{p}dx\right)
^{1/p}\left( \int_{a}^{b}(x-a)\left\vert g(x)\right\vert ^{q}dx\right)
^{1/q}\right\} \\
&\leq &\frac{1}{b-a}\left\{ \frac{\int_{a}^{b}(b-x)\left\vert
f(x)\right\vert ^{p}dx}{p\int_{a}^{b}\left\vert f(x)\right\vert ^{p}dx}+%
\frac{\int_{a}^{b}(b-x)\left\vert g(x)\right\vert ^{q}dx}{%
q\int_{a}^{b}\left\vert g(x)\right\vert ^{q}dx}\right. \\
&&\left. +\frac{\int_{a}^{b}(x-a)\left\vert f(x)\right\vert ^{p}dx}{%
p\int_{a}^{b}\left\vert f(x)\right\vert ^{p}dx}+\frac{\int_{a}^{b}(x-a)\left%
\vert g(x)\right\vert ^{q}dx}{q\int_{a}^{b}\left\vert g(x)\right\vert ^{q}dx}%
\right\} \\
&=&\frac{1}{p}+\frac{1}{q}=1.
\end{eqnarray*}%
This completes the proof.
\end{proof}

\begin{remark}
The inequality (\ref{2-2}) show that the inequality (\ref{2-1}) is better
than the inequality (\ref{1-1}).
\end{remark}

The more general versions of Theorem \ref{2.2} can be given as follow:

\begin{theorem}
\label{T-2.1} Let $p>1$ and $1/p+1/q=1$. If $f\ $and $g$ are real functions
defined on $\left[ a,b\right] $ and if $\left\vert f\right\vert
^{p},\left\vert g\right\vert ^{q}$ are integrable functions on $\left[ a,b%
\right] $, then

i.)%
\begin{eqnarray}
&&\int_{a}^{b}\left\vert f(x)g(x)\right\vert dx  \label{2-1-1} \\
&\leq &\left\{ \left( \int_{a}^{b}\alpha (x)\left\vert f(x)\right\vert
^{p}dx\right) ^{1/p}\left( \int_{a}^{b}\alpha (x)\left\vert g(x)\right\vert
^{q}dx\right) ^{1/q}\right.  \notag \\
&&\left. +\left( \int_{a}^{b}\beta (x)\left\vert f(x)\right\vert
^{p}dx\right) ^{1/p}\left( \int_{a}^{b}\beta (x)\left\vert g(x)\right\vert
^{q}dx\right) ^{1/q}\right\} ,  \notag
\end{eqnarray}%
where $\alpha ,\beta :\left[ a,b\right] \rightarrow \left[ 0,\infty \right) $
are continuous functions such that $\alpha (x)+\beta (x)=1,\ x\in \left[ a,b%
\right] .$

ii.) 
\begin{eqnarray*}
&&\int_{a}^{b}\left\vert f(x)g(x)\right\vert dx \\
&\leq &\sum_{i=1}^{n}\left( \int_{a}^{b}\alpha _{i}(x)\left\vert
f(x)\right\vert ^{p}dx\right) ^{1/p}\left( \int_{a}^{b}\alpha
_{i}(x)\left\vert g(x)\right\vert ^{q}dx\right) ^{1/q}
\end{eqnarray*}%
where $\alpha _{i}:\left[ a,b\right] \rightarrow \left[ 0,\infty \right)
,i=1,2,...n,$ are continuous functions such that $\sum_{i=1}^{n}\alpha
_{i}(x)=1,\ x\in \left[ a,b\right] .$
\end{theorem}

\begin{proof}
The proof of Theorem is easily seen by using similar method the proof of
Theorem \ref{2.2}.
\end{proof}

\begin{remark}
It is easily seen that the inequalities obtained in Theorem \ref{T-2.1} are
the best than the inequality (\ref{1-1}).
\end{remark}

\begin{remark}
i.) In the inequality (\ref{2-1-1}) of Theorem \ref{T-2.1}, if we take $%
\alpha (x)=\sin ^{2}x$ and $\beta (x)=\cos ^{2}x$, then we have 
\begin{eqnarray*}
&&\int_{a}^{b}\left\vert f(x)g(x)\right\vert dx \\
&\leq &\left\{ \left( \int_{a}^{b}\sin ^{2}x\left\vert f(x)\right\vert
^{p}dx\right) ^{1/p}\left( \int_{a}^{b}\sin ^{2}x\left\vert g(x)\right\vert
^{q}dx\right) ^{1/q}\right. \\
&&\left. +\left( \int_{a}^{b}\cos ^{2}x\left\vert f(x)\right\vert
^{p}dx\right) ^{1/p}\left( \int_{a}^{b}\cos ^{2}x\left\vert g(x)\right\vert
^{q}dx\right) ^{1/q}\right\} .
\end{eqnarray*}

ii.) In the inequality (\ref{2-1-1}) of Theorem \ref{T-2.1}, if we take $%
\alpha (x)=\frac{b-x}{b-a}$ and $\beta (x)=\frac{x-a}{b-a}$, then we have
the inequality (\ref{2-1}).
\end{remark}

\begin{theorem}
\label{2.3}Let $a=\left( a_{1},...,a_{n}\right) $ and $b=\left(
b_{1},...,b_{n}\right) $ be two positive n-tuples and $p,q>0$ such that $%
1/p+1/q=1.$ Then

\begin{itemize}
\item[i.)] 
\begin{eqnarray}
\sum_{k=1}^{n}a_{k}b_{k} &\leq &\frac{1}{n}\left\{ \left(
\sum_{k=1}^{n}ka_{k}^{p}\right) ^{1/p}\left( \sum_{k=1}^{n}kb_{k}^{q}\right)
^{1/q}\right.  \label{2-3} \\
&&\left. +\left( \sum_{k=1}^{n}\left( n-k\right) a_{k}^{p}\right)
^{1/p}\left( \sum_{k=1}^{n}\left( n-k\right) b_{k}^{q}\right) ^{1/q}\right\}
.  \notag
\end{eqnarray}

\item[ii.)] 
\begin{eqnarray}
&&\frac{1}{n}\left\{ \left( \sum_{k=1}^{n}ka_{k}^{p}\right) ^{1/p}\left(
\sum_{k=1}^{n}kb_{k}^{q}\right) ^{1/q}+\left( \sum_{k=1}^{n}\left(
n-k\right) a_{k}^{p}\right) ^{1/p}\left( \sum_{k=1}^{n}\left( n-k\right)
b_{k}^{q}\right) ^{1/q}\right\}  \label{2-4} \\
&\leq &\left( \sum_{k=1}^{n}a_{k}^{p}\right) ^{1/p}\left(
\sum_{k=1}^{n}b_{k}^{q}\right) ^{1/q}.  \notag
\end{eqnarray}
\end{itemize}
\end{theorem}

\begin{proof}
i.) By using of H\"{o}lder inequality in (\ref{1-2}), it is easily seen that%
\begin{eqnarray*}
&&\sum_{k=1}^{n}a_{k}b_{k} \\
&=&\sum_{k=1}^{n}\left( \frac{k}{n}+\frac{n-k}{n}\right) a_{k}b_{k} \\
&=&\frac{1}{n}\left\{
\sum_{k=1}^{n}k^{1/p}a_{k}k^{1/q}b_{k}+\sum_{k=1}^{n}\left( n-k\right)
^{1/p}a_{k}\left( n-k\right) ^{1/q}b_{k}\right\} \\
&\leq &\frac{1}{n}\left\{ \left( \sum_{k=1}^{n}ka_{k}^{p}\right)
^{1/p}\left( \sum_{k=1}^{n}kb_{k}^{q}\right) ^{1/q}+\left(
\sum_{k=1}^{n}\left( n-k\right) a_{k}^{p}\right) ^{1/p}\left(
\sum_{k=1}^{n}\left( n-k\right) b_{k}^{q}\right) ^{1/q}\right\} .
\end{eqnarray*}

ii.) First let us consider the case%
\begin{equation*}
\left( \sum_{k=1}^{n}a_{k}^{p}\right) ^{1/p}\left(
\sum_{k=1}^{n}b_{k}^{q}\right) ^{1/q}=0.
\end{equation*}

Then $a_{k}=0$ for $k=1,2,..,n$ or $b_{k}=0$ for $k=1,2,..,n.$ Thus, we have 
\begin{equation*}
\sum_{k=1}^{n}a_{k}b_{k}=0.
\end{equation*}%
Therefore the inequality (\ref{2-4}) is trivial in this case.

Finally, we consider the case%
\begin{equation*}
S=\left( \sum_{k=1}^{n}a_{k}^{p}\right) ^{1/p}\left(
\sum_{k=1}^{n}b_{k}^{q}\right) ^{1/q}\neq 0.
\end{equation*}%
Then%
\begin{eqnarray*}
&&\frac{1}{nS}\left\{ \left( \sum_{k=1}^{n}ka_{k}^{p}\right) ^{1/p}\left(
\sum_{k=1}^{n}kb_{k}^{q}\right) ^{1/q}+\left( \sum_{k=1}^{n}\left(
n-k\right) a_{k}^{p}\right) ^{1/p}\left( \sum_{k=1}^{n}\left( n-k\right)
b_{k}^{q}\right) ^{1/q}\right\} \\
&=&\frac{1}{n}\left\{ \left( \frac{\sum_{k=1}^{n}ka_{k}^{p}}{%
\sum_{k=1}^{n}a_{k}^{p}}\right) ^{1/p}\left( \frac{\sum_{k=1}^{n}kb_{k}^{q}}{%
\sum_{k=1}^{n}b_{k}^{q}}\right) ^{1/q}+\left( \frac{\sum_{k=1}^{n}\left(
n-k\right) a_{k}^{p}}{\sum_{k=1}^{n}a_{k}^{p}}\right) ^{1/p}\left( \frac{%
\sum_{k=1}^{n}\left( n-k\right) b_{k}^{q}}{\sum_{k=1}^{n}b_{k}^{q}}\right)
^{1/q}\right\} .
\end{eqnarray*}%
Applying (\ref{0-1}) on the right hand sides sums of the last inequality%
\begin{eqnarray*}
&&\frac{1}{nS}\left\{ \left( \sum_{k=1}^{n}ka_{k}^{p}\right) ^{1/p}\left(
\sum_{k=1}^{n}kb_{k}^{q}\right) ^{1/q}+\left( \sum_{k=1}^{n}\left(
n-k\right) a_{k}^{p}\right) ^{1/p}\left( \sum_{k=1}^{n}\left( n-k\right)
b_{k}^{q}\right) ^{1/q}\right\} \\
&\leq &\frac{1}{n}\left\{ \frac{\sum_{k=1}^{n}ka_{k}^{p}}{%
p\sum_{k=1}^{n}a_{k}^{p}}+\frac{\sum_{k=1}^{n}kb_{k}^{q}}{%
q\sum_{k=1}^{n}b_{k}^{q}}+\frac{\sum_{k=1}^{n}\left( n-k\right) a_{k}^{p}}{%
p\sum_{k=1}^{n}a_{k}^{p}}+\frac{\sum_{k=1}^{n}\left( n-k\right) b_{k}^{q}}{%
q\sum_{k=1}^{n}b_{k}^{q}}\right\} \\
&=&\frac{1}{p}+\frac{1}{q}=1.
\end{eqnarray*}%
This completes the proof.
\end{proof}

\begin{remark}
The inequality (\ref{2-4}) show that the inequality (\ref{2-3}) is better
than the inequality (\ref{1-2}).
\end{remark}

The more general versions of Theorem \ref{2.2} can be given as follow:

\begin{theorem}
\label{T-2.2} Let $a=\left( a_{1},...,a_{n}\right) $ and $b=\left(
b_{1},...,b_{n}\right) $ be two positive n-tuples and $p,q>0$ such that $%
1/p+1/q=1.$

i.) If $c=\left( c_{1},...,c_{n}\right) $ and $d=\left(
d_{1},...,d_{n}\right) $ be two positive n-tuples such that $%
c_{k}+d_{k}=1,k=1,2,...,n.$ Then 
\begin{eqnarray}
\sum_{k=1}^{n}a_{k}b_{k} &\leq &\left\{ \left(
\sum_{k=1}^{n}c_{k}a_{k}^{p}\right) ^{1/p}\left(
\sum_{k=1}^{n}c_{k}b_{k}^{q}\right) ^{1/q}\right.   \label{2-2-1} \\
&&\left. +\left( \sum_{k=1}^{n}d_{k}a_{k}^{p}\right) ^{1/p}\left(
\sum_{k=1}^{n}d_{k}b_{k}^{q}\right) ^{1/q}\right\} .  \notag
\end{eqnarray}%
ii.) If $c^{(i)}=\left( c_{1}^{(i)},...,c_{n}^{(i)}\right) ,i=1,2,...,m$ be
positive n-tuples such that $\sum_{i=1}^{m}c_{k}^{(i)}=1,k=1,2,...,n.$ Then 
\begin{eqnarray*}
\sum_{k=1}^{n}a_{k}b_{k} &\leq &\sum_{i=1}^{m}\left\{ \left(
\sum_{k=1}^{n}c_{k}^{(i)}a_{k}^{p}\right) ^{1/p}\left(
\sum_{k=1}^{n}c_{k}^{(i)}b_{k}^{q}\right) ^{1/q}\right\}  \\
&&
\end{eqnarray*}
\end{theorem}

\begin{proof}
The proof of Theorem is easily seen by using similar method the proof of
Theorem \ref{2.2}.
\end{proof}

\begin{remark}
It is easily seen that the inequalities obtained in Theorem \ref{T-2.2} are
the best than the inequality (\ref{1-2}).
\end{remark}

\begin{remark}
i.) In the inequality (\ref{2-2-1}) of Theorem \ref{T-2.2}, if we take $%
c=\left( \sin ^{2}1,...,\sin ^{2}n\right) $ and $d=\left( \cos
^{2}1,...,\cos ^{2}n\right) $, then we have 
\begin{eqnarray*}
\sum_{k=1}^{n}a_{k}b_{k} &\leq &\left\{ \left( \sum_{k=1}^{n}\sin
^{2}ka_{k}^{p}\right) ^{1/p}\left( \sum_{k=1}^{n}\sin ^{2}kb_{k}^{q}\right)
^{1/q}\right. \\
&&\left. +\left( \sum_{k=1}^{n}\cos ^{2}ka_{k}^{p}\right) ^{1/p}\left(
\sum_{k=1}^{n}\cos ^{2}kb_{k}^{q}\right) ^{1/q}\right\} .
\end{eqnarray*}

ii.) In the inequality (\ref{2-2-1}) of Theorem \ref{T-2.2}, if we take $%
c=\left( \frac{1}{n},\frac{2}{n}...,1\right) $ and $d=\left( \frac{n-1}{n},%
\frac{n-2}{n}...,0\right) $ , then we have the inequality (\ref{2-3}).
\end{remark}

\section{An Application}

In \cite{DA98}, Dragomir et al. gave the following lemma for obtain main
results.

\begin{lemma}
\label{L-3.1}Let $f:I^{\circ }\subseteq 
\mathbb{R}
\rightarrow 
\mathbb{R}
$ be a differentiable mapping on $I^{\circ }$, $a,b\in I^{\circ }$ with $a<b$
and $q>1.$ If $f\in L\left[ a,b\right] $, then the following equality holds:%
\begin{equation*}
\frac{f(a)+f(b)}{2}-\frac{1}{b-a}\int_{a}^{b}f(x)dx=\frac{b-a}{2}%
\int_{0}^{1}(1-2t)f(ta+(1-t)b)dt.
\end{equation*}
\end{lemma}

By using this equality and H\"{o}lder integral inequality Dragomir et al.
obtained the following inequality:

\begin{theorem}
\label{T-3.1}Let $f:I^{\circ }\subseteq 
\mathbb{R}
\rightarrow 
\mathbb{R}
$ be a differentiable mapping on $I^{\circ }$, $a,b\in I^{\circ }$ with $a<b$%
. If the new mapping $\left\vert f^{\prime }\right\vert ^{q}$\ is convex on $%
\left[ a,b\right] $, then the following inequality holds:%
\begin{equation}
\left\vert \frac{f(a)+f(b)}{2}-\frac{1}{b-a}\int_{a}^{b}f(x)dx\right\vert
\leq \frac{b-a}{2(p+1)^{1/p}}\left[ \frac{\left\vert f^{\prime
}(a)\right\vert ^{q}+\left\vert f^{\prime }(b)\right\vert ^{q}}{2}\right]
^{1/q},  \label{3-0}
\end{equation}%
where $1/p+1/q=1.$
\end{theorem}

If Theorem \ref{T-3.1} are resulted again by using the inequality (\ref{2-1}%
) in Theorem \ref{2.2}, then we get the following result:

\begin{theorem}
Let $f:I^{\circ }\subseteq 
\mathbb{R}
\rightarrow 
\mathbb{R}
$ be a differentiable mapping on $I^{\circ }$, $a,b\in I^{\circ }$ with $a<b$%
. If the new mapping $\left\vert f^{\prime }\right\vert ^{q}$\ convex on $%
\left[ a,b\right] $, then the following inequality holds:%
\begin{eqnarray}
&&\left\vert \frac{f(a)+f(b)}{2}-\frac{1}{b-a}\int_{a}^{b}f(x)dx\right\vert
\label{3-1-1} \\
&\leq &\frac{b-a}{4(p+1)^{1/p}}\left\{ \left[ \frac{2\left\vert f^{\prime
}(a)\right\vert ^{q}+\left\vert f^{\prime }(b)\right\vert ^{q}}{3}\right]
^{1/q}+\left[ \frac{\left\vert f^{\prime }(a)\right\vert ^{q}+2\left\vert
f^{\prime }(b)\right\vert ^{q}}{3}\right] ^{1/q}\right\} ,  \notag
\end{eqnarray}%
where $1/p+1/q=1.$
\end{theorem}

\begin{proof}
Using Lemma \ref{L-3.1} and the inequality (\ref{2-1}), we find%
\begin{eqnarray}
&&\left\vert \frac{f(a)+f(b)}{2}-\frac{1}{b-a}\int_{a}^{b}f(x)dx\right\vert
\label{3-1} \\
&\leq &\frac{b-a}{2}\int_{0}^{1}\left\vert 1-2t\right\vert \left\vert
f(ta+(1-t)b)\right\vert dt  \notag \\
&\leq &\frac{b-a}{2}\left\{ \left( \int_{0}^{1}(1-t)\left\vert
1-2t\right\vert ^{p}dt\right) ^{1/p}\left( \int_{0}^{1}(1-t)\left\vert
f(ta+(1-t)b)\right\vert ^{q}dt\right) ^{1/q}\right.  \notag \\
&&\left. +\left( \int_{0}^{1}t\left\vert 1-2t\right\vert ^{p}dt\right)
^{1/p}\left( \int_{0}^{1}t\left\vert f(ta+(1-t)b)\right\vert ^{q}dt\right)
^{1/q}\right\} .  \notag
\end{eqnarray}%
Using the convexity of $\left\vert f^{\prime }\right\vert ^{q}$, we have%
\begin{eqnarray}
\int_{0}^{1}t\left\vert f(ta+(1-t)b)\right\vert ^{q}dt &\leq &\int_{0}^{1}t 
\left[ t\left\vert f(a)\right\vert ^{q}+(1-t)\left\vert f(b)\right\vert ^{q}%
\right] dt  \notag \\
&=&\frac{2\left\vert f^{\prime }(a)\right\vert ^{q}+\left\vert f^{\prime
}(b)\right\vert ^{q}}{6},  \label{3-2}
\end{eqnarray}%
and%
\begin{eqnarray}
\int_{0}^{1}(1-t)\left\vert f(ta+(1-t)b)\right\vert ^{q}dt
&=&\int_{0}^{1}t\left\vert f(tb+(1-t)a)\right\vert ^{q}dt  \notag \\
&\leq &\frac{\left\vert f^{\prime }(a)\right\vert ^{q}+2\left\vert f^{\prime
}(b)\right\vert ^{q}}{6},  \label{3-3}
\end{eqnarray}%
Further, since%
\begin{eqnarray*}
\int_{0}^{1}t\left\vert 1-2t\right\vert ^{p}dt
&=&\int_{0}^{1}(1-t)\left\vert 1-2t\right\vert ^{p}dt \\
&=&\frac{1}{2(p+1)},
\end{eqnarray*}%
a combination of (\ref{3-1})-(\ref{3-3}) immediately gives the required
inequality (\ref{3-1-1}).
\end{proof}

\begin{remark}
Since $h:\left[ 0,\infty \right) \rightarrow 
\mathbb{R}
,h(x)=x^{s},0<s\leq 1,$ is a concave function, for \ all $u,v\geq 0$ we have%
\begin{equation*}
h\left( \frac{u+v}{2}\right) =\left( \frac{u+v}{2}\right) ^{s}\geq \frac{%
h(u)+h(v)}{2}=\frac{u^{s}+v^{s}}{2}.
\end{equation*}%
From here, we get%
\begin{equation*}
\frac{1}{2}\left[ \frac{2\left\vert f^{\prime }(a)\right\vert
^{q}+\left\vert f^{\prime }(b)\right\vert ^{q}}{3}\right] ^{1/q}+\frac{1}{2}%
\left[ \frac{\left\vert f^{\prime }(a)\right\vert ^{q}+2\left\vert f^{\prime
}(b)\right\vert ^{q}}{3}\right] ^{1/q}\leq \left[ \frac{\left\vert f^{\prime
}(a)\right\vert ^{q}+\left\vert f^{\prime }(b)\right\vert ^{q}}{2}\right]
^{1/q}
\end{equation*}%
Thus we obtain%
\begin{eqnarray*}
&&\frac{b-a}{4(p+1)^{1/p}}\left\{ \left[ \frac{2\left\vert f^{\prime
}(a)\right\vert ^{q}+\left\vert f^{\prime }(b)\right\vert ^{q}}{3}\right]
^{1/q}+\left[ \frac{\left\vert f^{\prime }(a)\right\vert ^{q}+2\left\vert
f^{\prime }(b)\right\vert ^{q}}{3}\right] ^{1/q}\right\} \\
&\leq &\frac{b-a}{2(p+1)^{1/p}}\left[ \frac{\left\vert f^{\prime
}(a)\right\vert ^{q}+\left\vert f^{\prime }(b)\right\vert ^{q}}{2}\right]
^{1/q}.
\end{eqnarray*}%
This show us that the inequality (\ref{3-1-1}) is the best than the
inequality (\ref{3-0}).
\end{remark}

\end{document}